\documentclass[12pt,reqno]{amsart}
\usepackage{dsfont}
\usepackage{amssymb}
\everymath{\displaystyle}
\newtheorem{theorem}{Theorem}[section]
\newtheorem{lemma}[theorem]{Lemma}
\newtheorem{proposition}[theorem]{Proposition}

\theoremstyle{definition}
\newtheorem{definition}[theorem]{Definition}

\newtheorem{remark}[theorem]{Remark}

\numberwithin{equation}{section}

 \newcommand{\IN}{\mathbb{N}}

 \newcommand{\IT}{\mathbb{T}}

\newcommand{\CA}{\mathcal{A}}
\newcommand{\CB}{\mathcal{B}}

\newcommand{\CD}{\mathcal{D}}

\newcommand{\CH}{\mathcal{H}}
\newcommand{\CJ}{\mathcal{J}}
\newcommand{\CL}{\mathcal{L}}
\newcommand{\CM}{\mathcal{M}}

\newcommand{\CS}{\mathcal{S}}
\newcommand{\CT}{\mathcal{T}}

\newcommand{\CK}{\mathcal{K}}

\newcommand{\ls}{(E,\mathcal{L}, \mathcal{B})}

\title{Primitive Ideals of Labelled Graph $C^*$-algebras}

\author{Menassie Ephrem}  
\address{Department of Mathematics and Statistics, Coastal Carolina
University, Conway, SC 29528--6054}
\email{menassie@coastal.edu}   



\begin{document}


\date{June, 2019.}


\begin{abstract}

Given a directed graph $E$ and a labeling $\CL$, one forms the labelled graph $C^*$-algebra by taking a weakly left--resolving labelled space  $\ls$ and considering a universal generating family of partial isometries and projections.  In this paper we provide characterization for primitive ideals of labelled graph $C^*$-algebras.
\end{abstract}

\maketitle


\section{Introduction} \label{sect.z21}

Since the work of Bratteli in the early 1970's, graphs have been used
as a tool to study a large class of $C^*$-algebras.  Bratteli
classified AF algebras in terms of their diagrams, later called
\textit{Bratteli diagrams}.  The current use of directed
graphs in $C^*$-algebras goes back to the work of Cuntz and Krieger
in \cite{CK}.  In that work, they associated a $C^*$-algebra to a
finite irreducible 0--1 matrix.

Later, it was noticed that one can view Cuntz--Krieger algebras as arising from graphs.  This approach of viewing Cuntz--Krieger algebras as
$C^*$-algebras associated to graphs made the construction more visual
and communicable.

In \cite{KPRR}, Kumjian, Pask, Raeburn and Renault defined
the graph groupoid of a countable row--finite directed graph with no
sinks, and showed that the $C^*$-algebra of this groupoid coincided
with a universal $C^*$-algebra generated by partial isometries
satisfying relations naturally generalizing those given in
\cite{Cu}.  Since that time, many people have worked on
generalizing these results to arbitrary directed graphs and beyond, including higher rank graphs, ultragraphs, and labelled graphs.

After the introduction of ultragraphs by Tomforde in \cite{T} Bates and Pask, in \cite{BP1}, introduced a new class of $C^*$-algebras called $C^*$-algebras of labelled graphs.  Later, in a serious of papers (along with Carlsen), they provided some classifications of these algebras, including computations of their $K$-theories.

A directed graph $E = (E^0, E^1,~~ s,~~ r)$ consists of a
countable set $E^0$ of vertices and $E^1$ of edges, and maps $s,r:
E^1 \rightarrow E^0$ identifying the source (origin) and the range
 (terminus) of each edge. The graph is row--finite if each
vertex emits at most finitely many edges. A vertex is a sink if it
is not a source of any edge.  A path is a
sequence of edges $e_1e_2\ldots e_n$ with $r(e_i) = s(e_{i+1})$
for each $i = 1,2,\ldots , n-1$. An infinite path is a sequence
$e_1e_2\ldots$ of edges with $r(e_i)=s(e_{i+1})$ for each $i$.

For a finite path $p=e_1e_2 \ldots e_n$, we define $s(p):=s(e_1)$
and $r(p):= r(e_n)$. For an infinite path $p=e_1e_2 \ldots$, we
define $s(p):=s(e_1)$. We use the following notations

$E^*:=\bigcup_{n=0}^\infty E^n,$ where $E^n$ := $\{p:p \text{ is a path of length } n\}.$

$E^{**}~~ :=~~ E^* \cup E^\infty$, where $E^\infty$ is the set of infinite paths.

The paper is organized as follows.  In section 2 we develop some terminologies for labelled graphs.  In section 3 we briefly describe labelled graph $C^*$-algebras.  In section 4, after building the tools needed and defining some properties of labelled spaces, we provide the theorems that characterize primitive ideals of a labelled graph $C^*$-algebra.  

\section{Preliminaries}\label{preliminaries}

Let $E = (E^0, ~E^1, ~s, ~r)$ be a directed graph and let $\CA$ be a set of alphabet (colors).  A labeling is a function $\CL : E^1 \longrightarrow \CA$.  Without loss of generality, we will assume that $\CA = \CL(E^1)$.  The pair $(E,~\CL)$ is called a labelled graph.

Given a labelled graph $(E, ~\CL)$, we extend the labeling function $\CL$ canonically to the sets $E^*$ and $E^\infty$ as follows.
Using $\CA^n$ for the set of words of size $n$, $\CL$ is defined from $E^n$ into $\CA^n$ as $\CL(e_1e_2\ldots e_n) = \CL(e_1)\CL(e_2)\ldots \CL(e_n)$.  Similarly, for $p = e_1e_2\ldots \in E^\infty$, $\CL(p) = \CL(e_1)\CL(e_2)\ldots \in \CA^\infty$.

Following a tradition, we use $\CL^*(E):=\bigcup_{n=1}^\infty \CL(E^n)$, and $\CL^\infty(E) := \CL(E^\infty)$.

For a word $\alpha = a_1a_2\ldots a_n \in \CL^n(E)$, we write $$s(\alpha) := \{s(p):p\in E^n,~ \CL(p) = \alpha\}$$ and $$r(\alpha) := \{r(p):p\in E^n,~ \CL(p) = \alpha\}.$$  Similarly for $\alpha = a_1a_2\ldots \in \CL^\infty(E)$, $$s(\alpha) := \{s(p):p\in E^\infty,~ \CL(p) = \alpha\}$$

Each of these sets is a subset of $E^0$.  The use of $s$ and $r$ for an edge/path verses a label/word should be clear from the context.

A labelled graph $(E,~\CL)$ is said to be left--resolving if for each $v\in E^0$ the function $\CL:r^{-1}(v) \rightarrow \CA$ is injective.  In other words, no two edges pointing to the same vertex are labelled the same.

Let $\CB$ be a non-empty subset of $2^{E^0}$.  Given a set $A \in \CB$ we write $\CL(AE^1)$ for the set $\{\CL(e): e \in E^1 \text{ and } s(a) \in A\}$.

For a set $A \in \CB$ and a word $\alpha \in \CL^n(E)$ we define the relative range of $\alpha$ with respect to $A$ as $$r(A,\alpha):= \{r(p): \CL(p) = \alpha \text{ and } s(p)\in A\}.$$

We say $\CB$ is closed under relative ranges if $r(A,\alpha) \in \CB$ for any $A\in\CB$ and any $\alpha \in \CL^n(E)$.

$\CB$ is said to be \textbf{\textit{accommodating}} if
\begin{enumerate}
    \item $r(\alpha)\in \CB$ for each $\alpha \in \CL^*(E)$
    \item $\CB$ is closed under relative ranges
    \item $\CB$ is closed under finite intersections and unions.
\end{enumerate}

If $\CB$ is accommodating for $(E, \CL)$, the triple $\ls$ is called a labelled space.  For trivial reasons, we will assume that $\CB \neq \{\emptyset\}$

A labelled space $\ls$ is called \textit{\textbf{weakly left--resolving}} if for any $A,~B \in \CB$ and any $\alpha \in \CL^*(E)$ $$r(A \cap B, \alpha) = r(A,\alpha) \cap r(B,\alpha).$$

We say $\ls$ is \textit{\textbf{non--degenerate}} if $\CB$ is closed under relative complements.  A \textit{\textbf{normal}} labelled space is accommodating and non--degenerate.

\section{Labelled Graph $C^*$-algebras}\label{preliminaries2}

Let $\ls$ be a weakly left--resolving labelled space.  A representation of $\ls$ in a $C^*$-algebra consists of projections $\{p_A : A \in \CB \}$,
and partial isometries $\{ s_a : a \in \CA\}$, satisfying:
\begin{enumerate}
\item If $A,~B \in \CB$, then $p_Ap_B = p_{A\cap B}$, and  $p_{A\cup B} = p_A+p_B-p_{A\cap B}$.

\item For any $a,~b \in \CA$, $s_a^*s_b = p_{r(a)}\delta_{a,b}$.

\item For any $a \in \CA$ and $A \in \CB$, $s_a^*p_A = p_{r(A,a)}s_a^*$.

\item For $A \in \CB$ with $\CL(AE^1)$ finite and $A$ does not contain a sink we have $$p_A = \sum_{a \in \CL(AE^1)}s_ap_{r(A,a)}s_a^*.$$

\end{enumerate}
The \textit{labelled graph $C^*$-algebra} is the $C^*$-algebra generated by a
universal representation of $\ls$.  For a word $\mu = a_1 \cdots
a_n$ we write $s_\mu$ to mean $s_{a_1} \cdots s_{a_n}$.  One easily checks
from the relations that $s_\mu^* s_\mu = p_{r(\mu)}$ and that $s_\nu^* s_\mu = 0$ unless one of $\mu$, $\nu$ extends the other.  In this case, e.g. if $\mu = \nu\alpha$, we have $s_\nu^* s_\mu = p_{r(\nu)} s_\alpha$.  

Using $\epsilon$ to denote the empty word, we find that
\[
C^*(E,\CL,\CB) = \overline{\text{span}} \{s_\mu p_A s_\nu^* : \mu,\;\nu \in \CL(E^*) \cup \{\epsilon\} \text{ and } A\in\CB \}.
\]

Here we use $s_\epsilon$ to mean the unit element of the multiplier algebra of $C^*(E,\CL,\CB)$.


Given a weakly left--resolving labelled space $\ls$, we say that the labelled space is set--finite (respectively receiver set--finite) if $s^{-1}(A)$ (respectively $r^{-1}(A)$) is finite for any $A \in \CB$.

We will assume that the graph $E$ has no sinks.

Notice that with this assumption if, in addition, $\ls$ is set--finite, then for any $A \in \CB$ we get $$p_A = \sum_{a \in \CL(AE^1)}s_ap_{r(A,a)}s_a^*.$$

\begin{definition}
Let $\ls$ be a labelled space.  For $A,~B \in \CB$, we say that ``$A$ sees $B$" and write $A \geq B$ if there exists $\alpha \in \CL^*(E)$ such that  $B \subseteq r(A,\alpha)$.
\end{definition}
Notice that this relation is transitive.

If $z \in \IT$, then the family $\{zs_a, p_A:a\in\CA, A\in\CB\}$ is another representation which
generates $C^*\ls$, and the universal property gives a homomorphism $\gamma_z : C^*\ls \rightarrow
C^*\ls$ such that $\gamma_z(s_a) = zs_a$ and $\gamma_z(p_A) = p_A$. The homomorphism $\gamma_{\overline{z}}$ is an inverse
for $\gamma_z$, so $\gamma_z \in \text{Aut} C^*\ls$, and an $\epsilon/3$ argument shows that $\gamma$ is
a strongly continuous action of $\IT$ on $C^*\ls$, called the \textit{gauge action}.

In \cite{JKP} they provided the definitions of ``hereditary" and ``saturated" that were adopted from the works of regular graphs (see \cite{BPRS}) and reworked to fit labelled graphs.

\begin{definition} \label{heredirary} For a subset $\CH$ of $\CB$ we say that $\CH$ is hereditary if it satisfies the following:
\begin{enumerate}
    \item for any $A\in \CH$ and for any $\alpha \in \CL^*(E)$ we have $r(A,\alpha) \in \CH$.
    \item $A \cup B \in \CH$ whenever $A,B \in \CH$.
    \item If $A \in \CH$ and $B\in \CB$ with $B \subseteq A$ then $B \in \CH$.
\end{enumerate}
\end{definition}

Notice that, in addition to being closed under finite unions, $\CH$ is closed under finite intersections.  Moreover, when $\ls$ is normal, if $A\in \CH$ and $B \in \CB$ then $A \setminus B \in \CH$.

\begin{definition}
A subset $\CH$ of $\CB$ is said to be saturated if for any $A \in \CB$, $\{r(A,a): a \in \CA\} \subseteq \CH$ implies that $A\in \CH$.
\end{definition}

It is easy to see that an arbitrary intersection of hereditary (respectively saturated) sets is hereditary (respectively saturated).  For a subset $\CH$ of $\CB$, we write $\overline{\CH}$ to mean the smallest hereditary and saturated subset of $\CB$ containing $\CH$, called the saturation of $\CH$.

For a subset $\CH$ of $\CB$ and $A \in \CB$, we write $A \triangleright \CH$ to mean $\{r(A,a):a\in \CA\} \subseteq \CH$.

Suppose $\CH$ is a hereditary subset of $\CB$ where $\ls$ is a weakly left--resolving, normal, set--finite labelled space.  If $A \triangleright \CH$ and $B\triangleright\CH$ then $A \cap B \triangleright \CH$ and $A \cup B \triangleright \CH$.  This is because $r(A\cap B, a) = r(A,a) \cap r(B,a)$ and $r(A\cup B, a) = r(A,a) \cup r(B,a)$ for any $a \in \CA$.

For a hereditary subset $\CH$ write $\CH_1 := \left\{A : A \triangleright \CH \right\}$, then $\CH_1$ is hereditary (and contains $\CH$).  Also, if $\CJ$ is a hereditary and saturated set containing $\CH$ then $\CH_1 \subseteq \CJ$.  Recursively define 
\begin{equation} 
\CH_{k+1} :=  \{A : A \triangleright \CH_k \}.\label{eq:1}
\end{equation}
We write $\CH_0$ for $\CH$.

\begin{lemma}
For a weakly left--resolving set--finite labelled space $\ls$, suppose $\CH \subseteq \CB$ is hereditary, let $\CK := \bigcup_{k = 0}^\infty \CH_k$.  Then $\overline{\CH} = \CK$.
\end{lemma}
\begin{proof}
Clearly $\CK$ is hereditary.  Let $A\in\CB$ and $a \in \CA$.   Since $\ls$ is set--finite, the set $\{a\in\CA : r(A,a) \neq \emptyset\}$ is finite.  Therefore, either $A \in \CH$ or $\exists n \in \IN$ such that $A \triangleright \CH_n$.  Which implies that $A \in \CH_{n+1}$, that is $A \in \CK$.  Therefore $\CK$ is saturated, hence $\overline{\CH} \subseteq \CK$. 

To prove  $\CK \subseteq \overline{\CH}$, we will first show that $\CH_1 \subseteq \overline{\CH}$.  Let $A \in \CH_1$.  Then $\{r(A,a):a\in \CA\} \subseteq \CH \subseteq \overline{\CH}$.  This implies that $A \in \overline{\CH}$ because $\overline{\CH}$ is saturated.  Hence $\CH_1 \subseteq \overline{\CH}$.   Similarly (inductively) $\CH_{n+1} \subseteq \overline{\CH}$.  Therefore $\CK \subseteq \overline{\CH}$.
\end{proof}
The above lemma provides us with a useful handle to the saturation of a hereditary subset of $\CB$.  We should note that this process does not apply if $\CH$ is an arbitrary (non hereditary) subset of $\CB$.
\begin{lemma}\label{lemm.BPRS62}
Suppose $\ls$ is a weakly left--resolving set--finite labelled space. Let $\CH$ be a hereditary and saturated subset of $\CB$ and let $\CD = \CB \setminus \CH$.  For a fixed $A\in \CD$, let $\CK = \{X \in \CD : A \geq X \}$. If $Y \in \overline{\CK}$ then there exist $Z\in \CK$ such that $Y \geq Z$. That is $A\geq Z$ and $Y \geq Z$.
\end{lemma}

\begin{proof}
First we note that $\CK$ is non--empty.  To see this, $\{r(A,a):a \in \CA\} \subseteq \CH \Longrightarrow A \in \CH$, since $\CH$ is saturated. However $A\in\CD$, thus $r(A,a_0) \notin \CH$ for some $a_0 \in \CA$.  Because $A \geq r(A,a_0)$ we have that $r(A,a_0) \in \CK$, i.e, $\CK$ is non--empty.  Observe also that $\CK$ is hereditary.  

If $Y \in \CK$ then $A \geq Y$; take $a \in \CL(YE^1)$, such that $r(Y,a) \notin \CH$. Then $A \geq r(Y,a)$ and $Y \geq r(Y,a)$, thus $Z = r(Y,a)$ will do.   Otherwise, $Y \in \CK_n$ for some $n \in \IN$, where $\CK_n$ is as in (\ref{eq:1}).  Take $\alpha = a_na_{n-1}\ldots a_1 \in \CL(YE^n)$. Notice that $r(Y,a_na_{n-1}\ldots a_k) \in \CK_{k-1}$.  Therefore $r(Y,\alpha) \in \CK$.  Taking $Z = r(Y,\alpha)$ concludes the proof.
\end{proof}

\section{Primitive ideals of $C^*\ls$} \label{sect.prim}

In this section we will provide characterisation of primitive ideals of $C^*\ls$.  The approach we use is similar to the methods used in \cite{BPRS}; some of the definitions and computations used there and in \cite{E} need to be remade and re--manufactured to fit the structure of labelled graphs and labelled graph $C^*$-algebras.

In \cite{JKP} they fully characterised the guage--invariant ideals of $C^*\ls$ in terms of the lattice of hereditary and saturated subsets of $\CB$.  Since their whole work is on weakly left--resolving, normal, set--finite, receiver set--finite labelled space on a graph that has no sinks, we will have the same assumptions.\\

\noindent
\textbf{Assumption:} For the rest of the paper we will assume that the graph $E$ has no sinks and that the labelled spaces we consider are weakly left--resolving, normal, set--finite, receiver set--finite labelled spaces.\\

We will list a couple of results from \cite{JKP} that relevant to our discussion.

\begin{lemma}  \label{Proposition.35JKP}\cite[Proposition 3.5]{JKP}
Let $I$ be a non--zero guage--invariant ideal of $C^*\ls$.  Then the relation $$A\sim_I B \Longleftrightarrow A\cup W = B\cup W, \text{ for some }W\in \CH.$$ defines an equivalence relation $\sim_I$ on $\CB$ such that $(E, \CL, [\CB]_I)$ is a weakly left--resolving quotient labelled space of $\ls$.
\end{lemma}

For a labelled space $\ls$ and a hereditary and saturated subset $\CH$ of $\CB$, write $I_\CH$ for the ideal of $C^*\ls$ generated by the set of projections $\{p_A:A\in\CH\}$.

\begin{lemma}\label{theorem52.jkp} \cite[Theorem 5.2]{JKP}
Let $I$ be a nonzero guage--invariant ideal of $C^*\ls$.  Then there exists an isomorphism of $C^*(E, \CL, [\CB]_I)$ onto the quotient algebra $C^*\ls/I$ and $I = I_\CH$, where $\CH$ is the hereditary and saturated subset consisting of $A\in\CB$ with $p_A \in I$.  Moreover the map $\CH \mapsto I_\CH$ gives an inclusion preserving bijection between the nonempty hereditary and saturated subsets of $\CB$ and the nonzero guage--invariant ideals of $C^*\ls$.

\end{lemma}

For a labelled space $\ls$ and an ideal $I$ of $C^*\ls$, write $\CH_I$ for the set $\{A\in \CB:p_A \in I\}$. In their proof of Lemma \cite[Lemma 5.1]{JKP}, for a hereditary and saturated set $\CH$, they showed that $\CH_{I_\CH} = \CH$.
\begin{remark}
Suppose that $\CH_1$ and $\CH_2$ are hereditary and saturated subsets of $\CB$ in a labelled space $\ls$.  As discussed earlier, $\CH_1 \cap \CH_2$ is also a hereditary and saturated subset of $\CB$.  If $A\in \CH_1$ then $A\cap B \in \CH_1$ for any $B\in \CB$, similarly for $\CH_2$.  Therefore $\CH_1\cap\CH_2 = \{A\cap B:A\in\CH_1 \text{ and } B\in \CH_2\}$.  

If $A\in \CH_1\cap \CH_2$ then $p_A$ is in $I_{\CH_1}$ also in $I_{\CH_2}$.  Therefore $I_{\CH_1\cap \CH_2} \subseteq I_{\CH_1} \cap I_{\CH_2}$. On the other hand $\CH_{I_{\CH_1} \cap I_{\CH_2}} =\{A\in \CB: p_A \in I_{\CH_1} \cap I_{\CH_2}\} \subseteq \CH_{I_{\CH_1}} \cap \CH_{I_{\CH_2}} = \CH_1\cap \CH_2$.  Therefore $I_{\CH_1} \cap I_{\CH_2} \subseteq I_{\CH_1\cap \CH_2}$.  That is $I_{\CH_1\cap \CH_2} = I_{\CH_1} \cap I_{\CH_2}$.

In fact, if $\{\CH_i\}$ is a collection of hereditary and saturated subsets of $\CB$ then $I_{\cap{\CH_\iota}} = \cap{I_{\CH_\iota}}$.
\end{remark}

In \cite{BPRS} they introduced the concept of a maximal tail of a directed graph to help characterize (the complement of) the set of vertices that are instrumental in providing primitive ideals of a graph $C^*$-algebra.  We will reformulate their definition to fit labelled graphs and prove similar results on primitive ideals of labelled graph $C^*$-algebras.  The phrase ``maximal tail", as defined for directed graphs, makes a lot more sense for directed graphs.

\begin{definition}\label{maxi.tail} Suppose $\ls$ is a labelled space.    A subset $\CD$ of $\CB$ is
called a maximal tail if it satisfies the following three conditions.

\begin{enumerate}

\item[(a)] for any $A, B \in \CD$ there exists $C \in \CD$
such that $A \geq C$ and $B \geq C$ .

\item[(b)] for any $A \in \CD$ there exists $a \in \CA$ such that $r(A,a) \in \CD$.

\item[(c)] $A \geq B$ and $B \in \CD$ imply that $A \in \CD$ .
\end{enumerate}
\end{definition}

We will prove a result similar to 
\cite[Lemma 3.2]{E} for labelled graph $C^*$-algebras.

\begin{proposition}\label{mainThm1.tail}  Let $\ls$ be a labelled space.  If $I$ is a primitive ideal of $C^*\ls$ and $\CH = \{A \in
\CB: p_A \in I\}$, then $\CD=\CB \setminus \CH$ is a maximal tail.
\end{proposition}

\begin{proof} by Lemma \ref{Proposition.35JKP} $\CH$ is hereditary and saturated.  

Since $E$ has no sinks, and $\CH$ is saturated,
$\CD$ satisfies (b). 

To prove (c), let $A \in \CB,~~ B \in \CD$ be such that $A \geq B$.  Then $\exists \alpha \in \CL^*(E)$ such that $B \subseteq r(A,\alpha)$.  We need to show that $A \in \CD$.  Assuming the contrary, if $A \in \CH$ then $r(A,\alpha) \in \CH$, that is $B \in \CH$, since $\CH$ is hereditary.  This contradicts to $B \in \CD$. Therefore $A \in \CD$. 

We prove (a).  Let $A,~~ B \in \CD$ and let $\CH_A = \{X \in \CD:A \geq X\}$ and $\CH_B = \{X \in \CD:B \geq X\}$.  We first show that $\overline{\CH_A} \cap \overline{\CH_B}
\neq \{\emptyset\}$.  

Consider $(E, \CL, [\CB]_I)$, the quotient space of $\ls$ under the relation $$A\sim_I B \Longleftrightarrow A\cup W = B\cup W, \text{ for some }W\in \CH.$$
Then $C^*(E, \CL, [\CB]_I) \cong C^*\ls/I_\CH$.  We claim that $[A] \neq [\emptyset]$.  Assuming the contrary, if $\exists W \in \CH$ such that $A\cup W = W\cup\emptyset \Rightarrow A\subseteq W \Rightarrow A \in \CH$ which is a contradiction to $A \in \CD$.  Similarly $[B] \neq [\emptyset]$.
Therefore $I_{\overline{\CH_A}}$ and $I_{\overline{\CH_B}}$ are both a non--zero ideals of $C^*(E, \CL, [\CB]_I) \cong C^*\ls/I_\CH$, hence they are of the form $I_A/I_\CH$ and $I_B/I_\CH$, so $p_{A} +
 I_\CH \in I_{\overline{\CH_A}}$ and  $p_{B} +
 I_\CH \in I_{\overline{\CH_B}}$. Since each $I_{\overline{\CH_i}}$ is
 gauge--invariant, so is $I_{\overline{\CH_A}}
 \cap I_{\overline{\CH_B}}$.  Therefore $I_{\overline{\CH_A}} \cap
 I_{\overline{\CH_B}} = I_{\overline{\CH_A} \cap \overline{\CH_B}}$.  If
$\overline{\CH_A} \cap \overline{\CH_B} = \{\emptyset\}$ then
$I_{\overline{\CH_A}} \cap I_{\overline{\CH_B}} = \{0\} \subseteq
I/I_\CH$.  But $I/I_\CH$ is a primitive ideal of $C^*\ls/I_\CH$
therefore $I_A/I_\CH \subseteq I/I_\CH$ or $I_B/I_\CH \subseteq I/I_\CH$.
Without loss of generality, let $I_A/I_\CH \subseteq I/I_\CH$ hence
$p_{A} + I_\CH \in I/I_\CH$ implying that $p_{A} \in I_\CH$ or
$p_{A} \in I \setminus I_\CH$. But $p_{A} \in I \setminus I_\CH$
is a contradiction to the construction of $\CH$, and $p_{A} \in
I_\CH$, which implies $A \in \CH$, is also a
contradiction to $A \in \CD = \CB \setminus \CH$. Therefore
$\overline{\CH_A} \cap \overline{\CH_B} \neq \{\emptyset\}$. Let $Y \in \overline{\CH_A} \cap \overline{\CH_B}$, $Y\neq \emptyset$. Applying Lemma
\ref{lemm.BPRS62} to $\CH$ and $A$ shows that there exists $Z \in
\CD~~ \textrm{such that}~~ Y \geq Z$ and $A \geq Z$.
Since $Y \in \overline{\CH_B}$ and $\overline{\CH_B}$ is hereditary,
$Z \in \overline{\CH_B}$. Applying Lemma \ref{lemm.BPRS62} to $\CH$ and $B$ shows that there exists $C \in \CD~~
\textrm{such that}~~ Z \geq C$ and $B \geq C$.  Thus $A \geq
C$ and $B \geq C$ as needed.
\end{proof}

In order to prove the converse of Proposition \ref{mainThm1.tail} we will prove one utility lemma.

\begin{lemma} \label{maxitaillemma}
Given a labelled space $\ls$ suppose $\CD \subsetneq \CB$  is a maximal tail and let $\CH = \CB \setminus \CD$.  Then $\CH$ is hereditary and saturated. 
\end{lemma}
\begin{proof}  We use roman numerals to label the paragraphs for reference within the proof.  The numbers (1), (2), (3) are in reference to Definition \ref{heredirary} and the labels (a), (b), (c) are in reference to Definition \ref{maxi.tail}.
\begin{enumerate}
    \item [(i)] Let $A \in \CB$.  If $\exists \alpha \in \CL^*(E)$ such that $r(A,\alpha) \notin \CH$ then $r(A,\alpha)\in\CD$.  But $A\geq r(A,\alpha)$, thus from (c) we get $A \in\CD$.  Therefore if $A\in\CH$ and $\alpha \in \CL^*(E)$ then $r(A,\alpha)\in\CH$.  This proves (1).

    \item[(ii)] To prove (3), let $A\in\CB$ and $B\subseteq A$.  Suppose $B\in\CD$ then by (b) $\exists a\in\CA$ such that $r(B,a) \in \CD$.  However $r(B,a) \subseteq r(A,a)$,  thus $A \geq r(B,a)$, implying that $A \in \CD$.  Hence if $A \in \CH$ and $B \subseteq A$ then $B \in \CH$.

    \item[(iii)] Now suppose $A, B \in\CH$.  Then by (ii) we have $A\cap B \in \CH$ and by (i) $r(A\cap B,a) \in \CH$ for each $a \in \CA$.  
If $A\cup B \in \CD$ then $\exists a\in \CA$ such that $r(A\cup B,a) \in \CD$.  However $r(A\cap B,a) \subseteq r(A\cup B,a)$ implying that $A\cup B \geq r(A\cap B,a)$, from (c) it follows that $r(A\cap B,a) \in \CD$; this is a contradiction.  Therefore $A\cup B\in\CH$.

\end{enumerate}

That $\CH$ is saturated follows from (b).
\end{proof}

Now we are ready to prove the converse of Proposition \ref{mainThm1.tail}.

\begin{proposition}\label{mainThm2.tail}
Given a labelled space $\ls$ suppose $\CD\subsetneq\CB$ is a maximal tail.  Let $\CH = \CB \setminus \CD$.  Then $I_\CH$ is a primitive ideal of $C^*\ls$.
\end{proposition}
\begin{proof}
Using Lemma \ref{maxitaillemma}, we see that $\CH$ is hereditary and saturated. To see that $I_\CH$ is a primitive ideal, it suffices to show that $I_\CH$ is prime. Suppose $I_1$, $I_2$ are ideals in $C^*\ls$ such that $I_1\cap I_2 \subseteq I_\CH$.  Then there are saturated sets $\CH_i$ such that $I_i = I_{\CH_i}$ and that $I_{\CH_1\cap \CH_2} = I_{\CH_1} \cap I_{\CH_2} \subseteq I_\CH$.  This implies $\CH_1\cap \CH_2 \subseteq \CH$.  If $\CH_1 \not\subset \CH$ and $\CH_2 \not\subset \CH$ then we can choose $A\in\CH_1\setminus\CH$ and $B\in\CH_2\setminus\CH$.  By (a) there exists $C\in\CD$ such that $A\geq C$ and $B\geq C$.  Then $C\in\CH_1\cap\CH_2\setminus\CH$; this contradicts $\CH_1\cap \CH_2 \subseteq \CH$.  Thus either $\CH_1 \subseteq \CH$ or $\CH_2 \subseteq \CH$ and $I_1 = I_{\CH_1} \subseteq I_\CH$ or $I_2 = I_{\CH_2} \subseteq I_\CH$.  This shows that $I_\CH$ is prime, and hence primitive.
\end{proof}

\begin{remark}
Proposition \ref{mainThm1.tail} and proposition \ref{mainThm2.tail} provide us a complete characterization of primitive ideals of $C^*\ls$.  However, given a labelled space $\ls$, one relevant question is ``how would one build a primitive ideal?''  It is generally easier to build a maximal tail and compute the complement, a hereditary and saturated subset $\CH$ of $\CB$, then construct a primitive ideal $I_\CH$.  This can be done by starting from a set, say $A$, and an infinite word, say $\alpha = a_1a_2 \ldots$, emanating from the set, then collect sets along the way to form $\CD_0 = \{A,~ r(A,a_1),~ r(A,a_1a_2),~ \ldots \}$.  Add in sets to form $\CD_{k+1} = \{A: r(A,a) \in \CD_k \text{ for some } a \in \CA\}$, finally take the union to get $\CD = \cup_{k\geq 0}\CD_k$, a maximal tail.
\end{remark}

For subsets $\CS, \CT$ of $2^{E^0}$, we write $\CS \gg \CT$ to mean that for each $A \in \CS$ there exist $B \in \CT$ such that $A \geq B$.  

We denote by $\chi_{\ls}$ the set of maximal tails in $\ls$.  We have formed a one--to--one correspondence between the set $\chi_{\ls}$ and the set $Prim C^*\ls$ of primitive ideals of $C^*\ls$.

\begin{remark}
Let $\CD \subseteq \CB$ be a maximal tail and let $A \in \CD$ then $r(A,a) \in \CD$ for some $a \in \CA$.  However $A \geq r(A,a)$, so $\CD \gg \CD$.
\end{remark}

\begin{theorem}
Let $\ls$ be a labelled space then there is a topology on the set $\chi_{\ls}$ of maximal tails in $\ls$ such that $\Phi: \chi_{\ls} \rightarrow Prim C^*\ls$ given by $\Phi(\CD) = I_{\CH_\CD}$ is a homeomorphism, where $\CH_\CD = \CB\setminus\CD$. 
\end{theorem}
\begin{proof}
Define a topology on $\chi_{\ls}$ by $$\overline{\xi} = \{\CD \in \chi_{\ls} : \CD \gg \bigcup_{\CT \in \xi}\CT\}$$ for $\xi \subseteq \chi_{\ls}$. 
The rest of the proof is a careful adaptation of the proof of  \cite[Theorem 6.3]{BPRS} that was done for directed graphs, with adjustments to fit labelled graphs. We will include it here for completion.

We verify that the operation $\xi \mapsto \overline{\xi}$ satisfies Kuratowski's closure axioms.  For the empty set, $\overline{\emptyset} = \emptyset$ is trivially true.  Also, if $\CD \in \xi$ then $\CD \gg \CD \Rightarrow \CD \gg \bigcup_{\CT \in \xi}\CT$, so $\xi \subseteq \overline{\xi}$.  We then have $\overline{\xi} \subseteq \overline{\overline{\xi}}$.
Let $\CD \in \overline{\overline{\xi}}$ and let $A \in \CD$ then there exists $\CT \in \overline{\xi}$ and $B \in \CT$ such that $A \geq B$.  However $\CT \in \overline{\xi}$ implies that there exists $\CM \in \xi$ and $C \in \CM$ such that $B \geq C$.  This gives us $A \geq C$.  Therefore $\CD \in \overline{\xi}$, that is $\overline{\overline{\xi}} \subseteq \overline{\xi}$.

Now let $\xi, \zeta \in \chi_{\ls}$.  Since $\xi \subseteq \xi \cup \zeta$, $\overline{\xi} \subseteq \overline{\xi \cup \zeta}$, similarly $\overline{\zeta} \subseteq \overline{\xi \cup \zeta}$.  Therefore $\overline{\xi} \cup \overline{\zeta} \subseteq \overline{\xi \cup \zeta}$.  

If $\CD \in \overline{\xi \cup \zeta}$ then for any $A \in \CD$ there exists $\CS \in \xi \cup \zeta$ and $B \in \CS$ such that $A \geq B$.  We will show that $\CS \in \xi$ or $\CS \in \zeta$, if we do that we get $\CD \in \overline{\xi}$ or $\CD \in \overline{\zeta}$. 

Define the sets $$\xi_\CS := \{X \in \CS: \{X\} \gg \bigcup_{\CT \in \xi}\CT\}$$ and $$\zeta_\CS := \{X \in \CS: \{X\} \gg \bigcup_{\CT \in \zeta}\CT\}.$$  Then $\CS = \xi_\CS \cup \zeta_\CS$.  We claim that $\CS = \xi_\CS$ or $\CS = \zeta_\CS$.  If not, choose $A \in \xi_\CS \setminus \zeta_\CS$ and $B \in \zeta_\CS \setminus \xi_\CS$. Since $A, B \in \CS$ we can choose $C \in \CS$ such that $A \geq C$ and $B \geq C$.  Then either $C \in \xi_\CS$ or $C \in \zeta_\CS$.  If $C \in \xi_\CS$ then $B \geq C$ and $\{C\} \gg \bigcup_{\CT \in \xi}\CT$, implying $B \in \xi_\CS$ and this is a contradiction.  Similarly for $C \in \zeta_\CS$.  Therefore $\CS \in \xi$ or $\CS \in \zeta$.  This gives us that $\CD \in \overline{\xi} \cup \overline{\zeta}$.  That is $\overline{\xi \cup \zeta} \subseteq \overline{\xi} \cup \overline{\zeta}$. Thus $\overline{\xi \cup \zeta} = \overline{\xi} \cup \overline{\zeta}$.  Therefore the operation $\xi \mapsto \overline{\xi}$ defines a topology on $\chi_{\ls}$.

What remains is to show that $\Phi(\overline{\xi}) = \overline{\Phi(\xi)}$.  Notice that if $\CD \gg \bigcup_{\CT \in \xi}\CT$ then for each $A \in \CD$ there exists $\CT \in \xi$ such that $A \in \CT$ this is because each $\CT$ is a maximal tail;  so $\CD \gg \bigcup_{\CT \in \xi}\CT$ implies that $\CD \subseteq \bigcup_{\CT \in \xi}\CT$.  Therefore 

\begin{align*}
\Phi(\overline{\xi}) &= \{I_{\CH_\CD}: \CD \subseteq \bigcup_{\CT \in \xi}\CT \}\\
&= \{I_{\CH_\CD}: \CH_\CD \supseteq \bigcap_{\CT \in \xi}\CH_\CT \}\\
&= \{I_{\CH_\CD}: I_{\CH_\CD} \supseteq I_{\bigcap_{\CT \in \xi}\CH_\CT} \}\\
&= \{I_{\CH_\CD}: I_{\CH_\CD} \supseteq \bigcap_{\CT \in \xi}I_{\CH_\CT} \}\\
&= \overline{\Phi(\xi)}.
\end{align*}
Therefore $\Phi$ is a homeomorphism.
\end{proof}

The assumption on the labelled space $\ls$ seems very restrictive, especially the requirements that it has to be set--finite, and that $E$ should have no sinks.  We believe that some of these restrictions can be relaxed if one recreates the works of Drinen, Tomforde and others (see \cite{DT}) similar to the idea of ``adding a tail".


\end{document}